\documentclass{amsart}
\usepackage{amsfonts,amssymb,latexsym,amsthm,amsmath}
\newtheorem{theorem}{Theorem}
\newtheorem{lemma}[theorem]{Lemma}
\newtheorem{remark}[theorem]{Remark}

\begin{document}
\title{A one-relator group with long lower central series}
\author{Roman Mikhailov}
\address{Chebyshev Laboratory, St. Petersburg State University, 14th Line, 29b,
Saint Petersburg, 199178 Russia and St. Petersburg Department of
Steklov Mathematical Institute} \email{rmikhailov@mail.ru}
\thanks{This research is supported by the Chebyshev
Laboratory  (Department of Mathematics and Mechanics, St.
Petersburg State University)  under RF Government grant
11.G34.31.0026, and by JSC "Gazprom Neft", as well as by the RF
Presidential grant MD-381.2014.1.}
\begin{abstract}
A one-relator group with lower central series of length $\omega^2$
is constructed. This answers a problem of G. Baumslag.
\end{abstract}
\maketitle

\section{Introduction} For a group $G$, the lower central series are defined
inductively as follows: $$\gamma_1(G)=G,\
\gamma_{\alpha+1}(G)=[\gamma_\alpha(G),G]$$ and
$\gamma_\tau(G)=\cap_{\alpha<\tau}\gamma_\alpha(G)$ for a limit
ordinal $\tau$. The smallest ordinal $\alpha$, such that
$\gamma_{\alpha}(G)=\gamma_{\alpha+1}(G)$ is called the {\it
length of the lower central series of $G$}. First examples of
finitely presented groups with the lower central series of length
greater than $\omega$ were constructed by J. Levine \cite{Levine}.
T. Cochran and K. Orr constructed examples of 3-manifold groups
with lower central series of length greater than $\omega$ in
\cite{CochranOrr}.

G. Baumslag asked the following question (\cite{Baumslag}, Problem
10): {\it Is the lower central series of a 1-relator group of
length at most $\omega n$ for some finite ordinal $n$?}

Consider the following one-relator group
$$
G=\langle a,b\ |\ a^{b^2}=aa^{3b}\rangle
$$
This group is a semidirect product $F_2\rtimes \mathbb Z$, where a
generator of the cyclic group acts on $F_2=F(x,y)$ as the
automorphism

\begin{align*}
& x\mapsto y\\
& y\mapsto xy^3
\end{align*}
The main result is the following
\begin{theorem}
The group $G$ has the lower central series of length $\omega^2$.
\end{theorem}

The group $G$ provides an example of a one-relator group with
length greater than $\omega n$ for any finite $n$, thus answering
the above problem of G. Baumslag.

\section{Proof of theorem 1}
Let $\Gamma$ be a group and $M$ a $\mathbb Z[\Gamma]$-module. $M$
is called {\it residually nilpotent} if
$$
\cap_kM\Delta^k(\Gamma)=0,
$$
where $\Delta^k(\Gamma)$ is the $k$th power of the augmentation
ideal $\Delta(\Gamma):=\ker\{\mathbb Z[\Gamma]\to \mathbb Z\}.$
\begin{lemma}
Let $M$ be a finitely generated free abelian group of rank $m\geq
1$ and an infinite cyclic group with generator $t$ acts on $M$ as
a matrix $A\in GL_m(\mathbb Z)$. Suppose that the product of any
collection of eigenvalues of $A-Id$ is not $\pm 1$. Then the
module $M$ is residually nilpotent $\mathbb Z[\langle
t\rangle]$-module.
\end{lemma}
\begin{proof}
Let $\Delta$ be the augmentation ideal of $\mathbb Z[\langle
t\rangle]$. Observe that, for any $n\geq 1$, the power $\Delta^n$
is the ideal generated by $(1-t)^n.$ Consider the submodule
$N:=M\Delta^\omega=\cap_{k\geq 1}M\Delta^k$ and assume that $N\neq
0.$ The Krull intersection theorem implies that $NI=N$, hence
$N(1-t)=N$ and the restriction of $(1-t)$ on $N$ defines a
bijection on $N$. Hence $det(A-Id)=\pm 1=\beta_1\dots \beta_k$,
where $(\beta_1,\dots, \beta_k)$ is some collection of eigenvalues
of $A-Id$. This gives a needed contradiction. Hence $N=0$.
\end{proof}

Now we prove theorem 1. As usual, if $x,y,a_1,\dots,a_{k+1}$ are
elements of a group $G$ we set $[x,y] =x^{-1}y^{-1}xy$,
$x^y=y^{-1}xy$ and define
$$[a_1,a_2,\dots, a_{k+1}]=[[a_1,\dots,a_{k}], a_{k+1}]\ (\, k>1).$$ For $n\geq 1$, define $T_n$ as a normal
closure in $G$ of all elements of the form $[x_1,\dots,x_n],$
where all $x_i$-s are either $a$ or $a^b$. We have
$$
H_1=\langle a\rangle^G=F(a,a^b),\ H_2=\langle [a,a^b]\rangle^G,\
H_3=\langle [a,a^b,a], [a,a^b,a^b]\rangle^G
$$
etc. Observe that
$$
H_{n}=\gamma_n(H_1).
$$
We claim that, for all $n\geq 1,$
\begin{equation}\label{omegan}
\gamma_{n\omega}(G)\subseteq H_{n+1}
\end{equation}

First we prove (\ref{omegan}) for $n=1$. Consider the quotient of
$G$ by the normal closure of the element $[a,a^b]:$
$$
H=\langle a,b\ |\ a^{b^2}=aa^{3b}, [a,a^b]=1\rangle
$$
The group $H$ is the metabelian polycyclic group $(\mathbb Z\oplus
\mathbb Z)\rtimes \mathbb Z$, where the generator of the cyclic
group $\mathbb Z$ acts on $\mathbb Z\oplus \mathbb Z$ as the
matrix
\begin{equation}\label{matr}
U:=\begin{pmatrix}
0 & 1 \\
1 & 3
\end{pmatrix}
\end{equation}
The eigenvalues of the matrix $U-Id$ are $
\frac{3\pm\sqrt{13}}{2}-1.$ By lemma 2, the group $H$ is
residually nilpotent. Lets show now that $[a,a^b]\in
\gamma_\omega(G)$. Observe that the element $a$ is a generalized
3-torsion element, i.e. for every $n$ there exists $t(n)$, such
that $a^{3^{t(n)}}\in\gamma_n(G)$. Indeed, since
\begin{equation}\label{3tor} a^3=[a,b^2]^{b^{-1}},\end{equation}
$$
a^{3^k}\in\gamma_{k+1}(G),\ k\geq 1.
$$
Therefore, the element $[a,a^b]$ also is a generalized 3-torsion
element. On the other hand, taking the commutator with $a^b$ of
the both sides of the relation, we get
$$
[a^{b^2},a^b]=[a,a^b]^{a^{3b}},
$$
hence $[a^b,a]=[a,a^b]^{a^{3b}b^{-1}}$ and
\begin{equation}\label{relation}
[a^b,a]^2=[a,a^b,a^{3b}b^{-1}]\end{equation}  Therefore $[a,a^b]$
is a generalized 2-torsion element. This implies that $[a,a^b]\in
\gamma_\omega(G)$. Since $H$ is residually nilpotent, we conclude
that $\gamma_\omega(G)=\langle [a,a^b]\rangle^G,$
$G/\gamma_\omega(G)=H$ and the statement (\ref{omegan}) is proved
for $n=1$. It is easy to see at this point that the group $G$ has
length greater than $\omega$. Indeed, since $H_2(H)=\mathbb Z/2$
and $H_2(G)=0$, we get
$$
H_2(H)=H_2(G/\gamma_\omega(G))\simeq
\gamma_{\omega}(G)/\gamma_{\omega+1}(G)\simeq \mathbb Z/2
$$
and
$$
[a,a^b]\in\gamma_{\omega}(G)\setminus\gamma_{\omega+1}(G),\
[a,a^b]^2\in\gamma_{\omega+1}(G).
$$

Now suppose that the statement (\ref{omegan})  holds for a given
$n$. Here we prove this statement for $n+1$.

We have an inclusion
$$
\gamma_{n\omega+k}(G)\subseteq [H_{n+1},\underbrace {G,\dots,
G}_k],\ k\geq 1
$$
Define the group $T:=G/H_{n+2}$. By the assumption of induction,
the term $\gamma_{n\omega}(T)$ lies in the quotient
$H_{n+1}/H_{n+2}$, the term $\gamma_{n\omega+k}(T)$ is a subgroup
of the quotient $([H_{n+1},\underbrace {G,\dots,
G}_k]H_{n+2})/H_{n+2}$ and can be presented as a product of
conjugates of the elements
$$
[\alpha, \underbrace {b,\dots, b}_k].H_{n+2},\ \alpha\in H_{n+1}.
$$
Consider the quotient $H_{n+1}/H_{n+2}$ as a $\mathbb Z[\langle
b\rangle]$-module. If the module $H_{n+1}/H_{n+2}$ is residually
nilpotent $\mathbb Z[\langle b\rangle]$-module, i.e.
$$
0=\cap_k (H_{n+1}/H_{n+2})\Delta^k,
$$
where $\Delta$ is the augmentation ideal of $\mathbb Z[\langle
b\rangle]$, then the intersection
$$
\cap_k([H_{n+1},\underbrace {G,\dots, G}_k]H_{n+2})/H_{n+2}
$$
is trivial in $T$ and the statement (\ref{omegan}) will follow for
$n+1$. Denote $M:=H_1/H_2.$ Since $H_1$ is a free group and
$H_i=\gamma_i(H_1),\ i\geq 1,$ the quotient $H_{n+1}/H_{n+2}$ is
the $n+1$-st Lie power $L^{n+1}(M)$ by Magnus-Witt theorem (see
\cite{Magnus}, \cite{Witt}). Lie powers are subgroups of tensor
powers $L^{n+1}(M)\subseteq \otimes^{n+1}(M)$ and therefore,
$$
\cap_k L^{n+1}(M)\Delta^k\subseteq \cap_k \otimes^{n+1}(M)\Delta^k
$$
Here we consider the tensor powers of $M$ as $\mathbb Z[\langle
b\rangle]$-modules with the diagonal action of $\langle b\rangle$.

For given matrices $A$ and $B$, the eigenvalues of their tensor
product $A\otimes B$ consist of all possible products of the
eigenvalues of $A$ and $B$. The eigenvalues of the matrix $U$ are
$\alpha_1=\frac{3+\sqrt{13}}{2},\ \alpha_2=\frac{3-\sqrt{13}}{2}.$
Hence, the eigenvalues of the the tensor power $U^{\otimes n+1}$
are $\pm \alpha_1^l$ and $\pm \alpha_2^s$ for some $l,s\geq 1$.
Suppose that one can present $\pm 1$ as a product of elements of
the form $\pm\alpha_1^l\pm 1$ and $\pm\alpha_2^s\pm 1$. Then the
product of their conjugates in the sense of quadratic
irrationality also is $\pm 1$. Since $\alpha_1^l\pm 1$ is the
conjugate quadratic irrationality to $\alpha_2^l\pm 1$, one can
present $\pm 1$ as a product of integers of the form
$$
(\alpha_1^l-1)(\alpha_2^l-1)\ \text{and}\
(\alpha_1^s+1)(\alpha_2^s+1)
$$
for different $l,s\geq 1$. Since $(\alpha_1-1)(\alpha_2-1)=-3$
divides $(\alpha_1^l-1)(\alpha_2^l-1)$ for any $l\geq 1$, the
absolute values of the terms $(\alpha_1^l-1)(\alpha_2^l-1)$ are
greater than 1. For an odd $s$, the product
$(\alpha_1^s+1)(\alpha_2^s+1)$ divides
$(\alpha_1+1)(\alpha_2+1)=3.$ For an even $s$, the terms
$(\alpha_1^s+1)$ and $(\alpha_2^s+1)$ are greater than 1. We get a
contradiction with assumption that one can present $\pm 1$ as a
product of elements of the form $\pm\alpha_1^l\pm 1$ and
$\pm\alpha_2^s\pm 1$.

We conclude that, for any $n$, the matrix $U^{\otimes n+1}-Id$
satisfies the conditions of lemma 2 and therefore the $(n+1)$st
tensor power $M^{\otimes n+1}$ is a residually nilpotent $\mathbb
Z[\langle b\rangle]$-module. This implies that the Lie power
$L^{n+1}(M)$ also is a residually nilpotent $\mathbb Z[\langle
b\rangle]$-module and the inductive step is done,
$$
\gamma_{(n+1)\omega}(G)\subseteq H_{n+2},
$$
proving (\ref{omegan}) for all $n$.

Since $H_1$ is a free group, it is residually nilpotent, and we
get
$$
\gamma_{\omega^2}(G)=\cap_n\gamma_{n\omega}(G)\subseteq
\cap_nH_{n+1}=\cap_n\gamma_{n+1}(H_1)=1.
$$
Since $[a^b,a]\in \gamma_{\omega}(G),$ we have
$$
1\neq [[a^b,a],a,\underbrace {[a^b,a],\dots, [a^b,a]}_k]\in
\gamma_{(k+1)\omega}(G),\ k\geq 1.
$$
This element clearly is non-trivial, since it is a basic
commutator in the basic elements $a,a^b$ of the free group
$H_1=F(a,a^b)$. Hence, the lower central series length of $G$ is
$\omega^2$ and theorem 1 is proved.

\begin{remark}
Observe that one can easily use the proof of theorem 1 to prove
the following more general result. Let $F_n$ be a free group and
an element $t$ acts on $F_n$ as an automorphism $\phi:F_n\to F_n$.
Define $\Phi=F_n\rtimes \langle t\rangle$. Suppose that any tensor
power $(F_{ab})^{\otimes k}\ (k\geq 1)$ is a residually nilpotent
$\mathbb Z[\langle t\rangle]$-module. Then
$\gamma_{\omega^2}(\Phi)=1$. In particular, if $\Phi$ is not
residually nilpotent, then the lower central series length of
$\Phi$ is exactly $\omega^2$.
\end{remark}

We add the next remark here due to its applications in the theory
of localizations \cite{Mikha}.
\begin{remark} The $2\omega$th lower central series
term of $G$ can be described as follows
\begin{equation}\label{2o}\gamma_{2\omega}(G)=\langle
[a,a^b,a],[a,a^b,a^b]\rangle^G\end{equation}
\end{remark}

We proved before that there is an inclusion
$\gamma_{2\omega}(G)\subseteq H_3.$ To see the converse inclusion,
observe that $a$ and $a^b$ are generalized 3-torsion elements,
hence $$[a,a^b,a]^{3^l}, [a,a^b,a^b]^{3^l}\subseteq
\gamma_{\omega+l}(G),\ l\geq 1.$$ However, the same is true for
powers of 2, due to (\ref{relation}). Therefore, $H_3\subseteq
\gamma_{2\omega}(G)$ and description (\ref{2o}) follows.

\vspace{.5cm}\noindent {\it Acknowledgements.} The author thanks
G. Baumslag, S.O. Ivanov, F. Petrov and K. Orr for discussions
related to the subject of the paper.

\end{document}